\documentclass[preprint,11pt,3p]{elsarticle}

\usepackage{amsmath}
\usepackage{amssymb}

\usepackage{amsthm}

\newtheorem{thm}{Theorem}[section]

\newtheorem{lem}[thm]{Lemma}

\theoremstyle{definition}
\newtheorem{defn}[thm]{Definition}
\theoremstyle{remark}
\newtheorem{rem}[thm]{Remark}
\numberwithin{equation}{section}

\newcommand{\f}{\Lambda(f)}
\newcommand{\Fq}{\mathbb{F}_{q}[x]}
\begin{document}

\begin{frontmatter}

\title{An improved asymptotic formula for the distribution of irreducible polynomials in arithmetic progressions over $\mathbb{F}_{q}[x]$}

\author[label1]{Zihan Zhang}
\address[label1]{Department of Mathematics, Sichuan University, Chengdu, 610065, China}
\ead{zzhsdj@foxmail.com}

\author[label5]{Dongchun Han\corref{cor1}}
\address[label5]{School of Mathematics, Southwest Jiaotong University, Chengdu, 610031, China}
\ead{han-qingfeng@163.com}

\cortext[cor1]{Corresponding author}

\tnotetext[fn1]{D. Han's research was supported by the National Science Foundation of China Grant No.11601448 and the Fundamental Research Funds for the Central Universities of China under Grant 2682016CX121.}

\begin{abstract}
Let $\mathbb{F}_{q}$ be a finite field with $q$ elements and $\mathbb{F}_{q}[x]$ the
ring of polynomials over $\mathbb{F}_{q}$. Let $l(x), k(x)$ be coprime polynomials in $\mathbb{F}_{q}[x]$ and $\Phi(k)$ the Euler function in $\mathbb{F}_{q}[x]$. Let $\pi(l, k; n)$ be the number of monic irreducible polynomials of degree $n$ in $\mathbb{F}_{q}[x]$ which are congruent to $l(x)$ module $k(x)$. For any positive integer $n$, we denote by $\Omega(n)$ the least prime divisor of $n$. In this paper, we show that
$$\pi(l, k; n)=\frac{1}{\Phi(k)}\frac{q^{n}}{n}+O\left(n^{\alpha}\right)+O\Big(\frac{q^{\frac{n}{\Omega{(n)}}}}{n}\Big),$$
where $\alpha$ only depends on the choice of $k(x)\in\Fq$. Note that the above error term improves the one implied by a deep result of A. Weil \cite{p}. Our approach is completely elementary.
\end{abstract}


\end{frontmatter}


\section{Introduction}

The origin of analytic number theory can be traced back to Euler's analytic proof of the existence of infinitely many primes in 1737. A century later, a seminal work by Dirichlet \cite{a}, inspired by Euler's proof, states that for two coprime integers $a$ and $b$, there are infinitely many primes $p$ such that $p\equiv a(\mathrm{mod}~  b)$. These classical results and methods introduced therein greatly stimulated the development of number theory.

Let $\mathbb{F}_{q}$ be a finite field with $q$ elements. We denote by $\mathbb{F}_{q}[x]$ the
ring of polynomials over $\mathbb{F}_{q}$. In 1919, H. Kornblum \cite{b} considered the analogue of Dirichlet's theorem over $\mathbb{F}_{q}[x]$ in his doctoral thesis. In fact, he proved that, given two coprime polynomials $l(x)$ and $k(x)$ in $\Fq$, there are infinitely many monic irreducible polynomials $p(x)\in\Fq$ such that $p(x)\equiv l(x)(\mathrm{mod}~ k(x))$.

Let $\pi(l, k; n)$ denote the number of monic irreducible polynomials of degree $n$ in $\mathbb{F}_{q}[x]$ which are congruent to $l(x)$ module $k(x)$, where polynomials $l(x),k(x)$ are coprime. We denote by $\Phi(k)$ the analogue of the Euler function which is defined as the size of the multiplicative group $\left(\Fq/(k(x))\right)^{\times}$.
In 1924, Artin \cite{o} refined Kornblum's result by showing the following theorem.
\begin{thm}[\cite{o}]\label{0}
 Let $l(x), k(x)$ be coprime polynomials
in $\Fq$. Then, it follows that
 \[\pi(l, k; n)=\frac{1}{\Phi(k)}\frac{q^{n}}{n}+O\Big(\frac{q^{n\theta}}{n}\Big) \quad \text { with } \quad \frac{1}{2} \leq \theta<1.\]

\end{thm}

Instead of considering the distribution of monic irreducible polynomials in arithmetic progressions, based on Dickson's work \cite{4} in 1911, Carlitz \cite{1} and Uchiyama \cite{2} provided some different results for distribution of monic irreducible polynomials in $\Fq$ with some given coefficients.
\begin{thm}[\cite{2}]\label{4}
Let $s, t$ be two positive integers and $S(s, t; n)$ the number of monic irreducible polynomials of degree $n$ in $\Fq$ with first $s$ coefficients
	and last $t$ coefficients given. Assume that the characteristic of $\mathbb{F}_{q}$ is larger than $\max\{s,t-1\}$. Moreover, if the last coefficient is not zero. Then we have
	\[S(s,t;n)=\frac{1}{q^{s+t-1}{(q-1)}}\frac{q^{n}}{n}+ O\Big(\frac{q^{n\theta}}{n}\Big)\quad \text { with } \quad \frac{1}{2} \leq \theta<1.\]
\end{thm}
Note that, when $s=t=1$ in the above theorem, Carlitz \cite{1} gave an asymptotic formula with a better error term.

Nevertheless, it seems that it is still not clear whether we can improve the error term of Theorem \ref{4} to $O\big(q^{\frac{n}{2}}/n\big)$.
Later, inspired by the results of Artin \cite{o} and Uchiyama \cite{2}, Hayes \cite{3} provided the following more general result.
\begin{thm}[\cite{3}]\label{5}
Let $s$ be an positive integer and $\pi_{S}(s,l,k;n)$ the number of monic irreducible polynomials of degree $n$ in $\Fq$ with given first $s$ coefficients which are congruent to $l(x)$ module $k(x)$. Then we have
	\[\pi_{S}(s,l,k;n)=\frac{1}{q^{s}{\Phi(k)}}\frac{q^{n}}{n}+ O\Big(\frac{q^{n\theta}}{n}\Big)\quad \text { with } \quad \frac{1}{2} \leq \theta<1,\]
where polynomials $l(x),k(x)$ are coprime.
\end{thm}
Although Hayes claimed that he can improve the error term to $O\big(q^{\frac{n}{2}}/n\big)$ when $x^{q}-x$ does not divide $k(x)$, he never published this result in detail. While, although still weaker than $O\big(q^{\frac{n}{2}}/n\big)$, Hensley \cite{c} provided some improved error terms in some special cases over $\mathbb{F}_{2}[x]$. In 2002, Rosen \cite{j} provided a modern perspective about the distribution of irreducible polynomials in arithmetic progressions over $\mathbb{F}_{q}[x]$ by showing the following theorem.
\begin{thm}[\cite{j}]\label{40}
Let $l(x), k(x)$ be coprime polynomials
in $\Fq$. Then, it follows that
 \[\pi(l,k;n)=\frac{1}{\Phi(k)}\frac{q^{n}}{n}+O\Big(\frac{q^{\frac{n}{2}}}{n}\Big).\]
\end{thm}
Note that, Rosen's proof is based on a deep result of A. Weil \cite{p}. Also, we refer to \cite{Yucas, Granger, Ha, Mullen, Hsu, Pollack, Mullen1} for recent progresses on the distribution of irreducible polynomials over $\mathbb{F}_{q}[x]$ and the references therein. In this paper, we shall significantly improve Theorem \ref{40}. For any positive integer $n$, denote by $\Omega(n)$ the least prime divisor of $n$. The following theorem is our main result.
\begin{thm}\label{1}
Let $l(x), k(x)$ be coprime polynomials in $\Fq$. Then, it follows that
 \[\pi(l,k;n)=\frac{1}{\Phi(k)}\frac{q^{n}}{n}+O\Big(n^{\alpha}\Big)+O\Big(\frac{q^{\frac{n}{\Omega(n)}}}{n}\Big),\]
 where  $\alpha$ is a constant only depends on the choice of $k(x)\in\Fq$.
\end{thm}
Note that the error term we obtain is better than the one that Weil's result provides. Our approach is completely elementary.
\section[Title for the table of contents]{Notation and preliminaries}

The following definitions and lemmas can be found in \cite{j}. Also, it should be noted that $f,k,m,h$ are always referred as monic polynomials in $\Fq$, $l$ is a polynomial in $\Fq$ and $p$ is always referred as the monic irreducible polynomial. For any positive integer $n$, denote by $\Omega(n)$ the least prime divisor of $n$.
\begin{defn}
The Riemann zeta function over $\mathbb{F}_{q}[x]$ is defined as:
\[\zeta(s)=\sum_{f \text{ monic } }{\frac{1}{|f|^{s}}} ,\]
where for $f\in\Fq$ we define $|f|=q^{\deg f}.$	
\end{defn}
\begin{defn}
For $f\in\mathbb{F}_{q}[x]$, the M\"{o}bius function is defined as:
\[\mu(f)=\left\{
\begin{array}{rcl}
1           &    & {\text{if } f=1},\\
(-1)^{r}    &    & {\text{if } f=p_{1}p_{2}\cdots p_{r}, \text{ } p_{i}\neq p_{j} \text{ for } i\neq j\text{ and }r\geq1},\\
0           &    & {otherwise.}
\end{array} \right.
\]
\end{defn}
\begin{lem}\label{7}
The Dirichlet $L$-series associated with $\mu(f)$ satisfies that
\[L(s,\mu):=\sum_{f \text{ monic } }{\frac{\mu(f)}{|f|^{s}}}:=\sum_{n=0}^{\infty}{\frac{H(n)}{q^{ns}}}=\frac{1}{\zeta(s)}.\]
As a consequence, we have
\[H(n)=\sum_{\substack{f \text{ monic } \\ \deg{f}=n}}{\mu(f)}=\left\{
\begin{array}{rcl}
1           &    & {\text{if } n=0},\\
-q          &    & {\text{if } n=1},\\
0           &    & {otherwise.}
\end{array} \right.\]
\end{lem}
\begin{defn}\label{17}
For $f\in\mathbb{F}_{q}[x]$, the von Mangolt's function is defined as:
   \[\f=\left\{
\begin{array}{rcl}
\deg p      &    & {\text{if } f=p^{\alpha},\text{ } \alpha\geq1,}\\
0           &    & {otherwise.}
\end{array} \right.\]
\end{defn}
 \begin{lem}[\cite{j}]\label{14}
  For any $n\geq 1$, let $\pi(n)$ denote the number of the monic irreducible polynomials in $\mathbb{F}_{q}[x]$ of degree $n$. Then, it follows that
 \[\sum_{\substack{f \text{ monic } \\ \deg{f}=n }}{\f}=\sum_{d|n}{d\pi(d)}= q^{n}.\]
 \end{lem}
\section[Title for the table of contents]{Proof of the Theorem \ref{1}}
\subsection{Some auxiliary lemmas}
The essential idea of ours is using $\Lambda(f)$ to sieve the power of irreducible polynomials. We consider the following summation
$$\sum_{\substack{f \text{ monic } \\ \deg{f}=n \\ f\equiv l (k)}}{\Lambda(f)}.$$  Summing it by two ways, one is to show it is very
close to $n\pi(l,k;n)$ and the other is to show that the summation is independent of $l\in(\mathbb{F}_{q}[T]/(k))^{\times}$. Therefore, we can accomplish Theorem \ref{1}.

We list some lemmas which are useful in our proof.

\begin{lem}\label{6}
Let $\pi(n)$ denote the number of the monic irreducible polynomials in $\mathbb{F}_{q}[x]$ of degree $n$. Then, it follows that
\[\pi(n)=\frac{q^{n}}{n}+O\Big(\frac{q^{\frac{n}{\Omega(n)}}}{n}\Big).\]

\end{lem}
\begin{proof}
At first, from Lemma \ref{14}, it follows that
\[\pi(n)=\frac{1}{n}\sum_{d\mid n}{\mu(d)q^{\frac{n}{d}}},\]
where $\mu(d)$ is the usual M\"{o}bius function defined on $\mathbb{N}.$ Therefore, it follows that
\[\pi(n)-\frac{q^{n}}{n}=\frac{1}{n}\sum_{\substack{d\mid n\\d\neq1}}{\mu(d)q^{\frac{n}{d}}}.\]
Note that it suffices to consider the case when n has at least three different prime divisors. We may assume that $n=p_{1}^{\alpha_{1}}p_{2}^{\alpha_{2}}\dots p_{t}^{\alpha_{t}}$ with
$2\leq p_{1}<\dots<p_{t}$ and $t\geq3$ as well as $\alpha_{i}\geq1$ for $1\leq i\leq t$. We denote $j=p_{1}^{\alpha_{1}-1}p_{2}^{\alpha_{2}-1}p_{3}^{\alpha_{3}}\dots p_{t}^{\alpha_{t}}$, then $\Omega(n)=p_{1}$ and
\[\begin{split}
 \Big|\frac{1}{n}\sum_{\substack{d\mid n\\d\neq1}}{\mu(d)q^{\frac{n}{d}}}\Big|\leq \frac{q^{\frac{n}{\Omega(n)}}}{n}+q^{jp_{1}}\leq \frac{2q^{\frac{n}{\Omega(n)}}}{n}.
\end{split}
\]
Note that, the last inequality comes from the following,

$$nq^{j{p_{1}}}\leq j^{3}q^{jp_{1}}\leq q^{j(p_{1}+2)}\leq q^{jp_{2}}.$$
This completes the proof.
\end{proof}
\begin{lem}\label{12}
  For any polynomials $l,m,k$ in $\mathbb{F}_{q}[x]$ satisfing $(l,k)=(m,k)=1$, we have
\[\sum_{\substack{f \text{ monic }\\ \deg{f}=n \\m|f\\ f\equiv l (k)}}{1}=\left\{
\begin{array}{rcl}
q^{n-\deg{k}-\deg{m}}          &    & {\text{if }\deg{m}\leq n-\deg{k},}\\
O(1)                           &    & {\text{if }\deg{m}> n-\deg{k}.}\\
\end{array} \right.
\]
\end{lem}
\begin{proof}
We have
$$\sum_{\substack{f \text{ monic }\\ \deg{f}=n \\m|f\\ f\equiv l (k)}}{1}=\sum_{\substack{g \text{ monic }\\ \deg{g}=n-\deg{m} \\ g\equiv l\bar{m} (k)}}{1},$$
 since $(l,k)=(m,k)=1$ and there is an $\bar{m}\in(\mathbb{F}_{q}[T]/(k))^{\times}$ such that $\bar{m}m\equiv1(\mathrm{mod}~k) $. Moreover, we may assume that $\deg{l\bar{m}}<\deg{k}$ and $g=l\bar{m}+hk$.

On one hand, when $\deg{m}\leq n-\deg{k}$, we have
\[\sum_{\substack{g \text{ monic }\\ \deg{g}=n-\deg{m} \\ g\equiv l\bar{m} (k)}}{1}=\sum_{\substack{h \text{ monic }\\ \deg{h}=n-\deg{m}-\deg{k} }}{1} = q^{n-\deg{k}-\deg{m}}.\]
On the other hand, when $\deg{m}> n-\deg{k}$, we have $h=0$. Therefore, $$\sum_{\substack{g \text{ monic }\\ \deg{g}=n-\deg{m} \\ g\equiv l\bar{m} (k)}}{1}\leq 1$$
for fixed polynomials $l,k$, that is to say
$$\sum_{\substack{g \text{ monic }\\ \deg{g}=n-\deg{m} \\ g\equiv l\bar{m} (k)}}{1}=\sum_{\substack{g \text{ monic }\\  \deg{g}=n-\deg{m}  \\ g= l\bar{m} }}{1}=O(1).$$
\end{proof}
\begin{rem}
If $\deg{l\bar{m}}<\deg{k}$ does not hold, we can always get $l^{'}\equiv\bar{m}(\mathrm{mod}~k)$ satisfing $\deg{l^{'}}<\deg{k}.$
\end{rem}
We also introduce the following $L$-series which plays an important role in our proof
\[L_k(s, \mu):=\sum_{\substack{f \text{ monic}\\(f,k)=1}}{\frac{\mu(f)}{|f|^{s}}}:=\sum_{N=0}^{\infty}{\frac{H(n,k)}{q^{ns}}},\]
and we list some of the following related results.
\begin{lem}\label{8}
For $\text{Re}(s)>1,$ we have that $$L_k(s, \mu)= L(s,\mu)\prod_{p|k}\frac{1}{1-|p|^{-s}}.$$
\end{lem}
\begin{proof}
It follows from Lemma \ref{7}.
\end{proof}
\begin{lem}\label{9}
Suppose that for a monic irreducible polynomial $k(x)\in\mathbb{F}_{q}[x]$ with $\deg{k}\geq2,$ we have
\[H(n,k)=\left\{
\begin{array}{rcl}
1          &    & {\text{if  } n\equiv 0(\mathrm{mod}~\deg{k}),}\\
-q         &    & {\text{if  } n\equiv 1(\mathrm{mod}~\deg{k}),}\\
0          &    & {otherwise.}\\
\end{array} \right.\]
\end{lem}
\begin{proof}
Applying Taylor expansion of both sides of the formula in Lemma \ref{8}, together with Lemma \ref{7}, we can obtain the result by comparing coefficients of both sides.
\end{proof}
\begin{lem}\label{10}
Let $k(x)=p_{1}^{\alpha_{1}}p_{2}^{\alpha_{2}}\ldots p_{t}^{\alpha_{t}}$, where
$p_{1}, p_{2}, \ldots , p_{t}$ are irreducible polynomials over $\mathbb{F}_{q}$ and $t$ is a positive integer, then we have
  \[|H(n,k)|\leq qn^{t-1}\]
  for $n\in\mathbb{N}.$
\end{lem}
\begin{proof}
We prove this lemma by induction on $t$. First, when t=1, we can verify the result easily by Lemma \ref{9}.
We assume that the desired result holds for $t=r$, then we consider the case when $t=r+1$. Since $k(x)=p^{\alpha} p_{1}^{\alpha_{1}}p_{2}^{\alpha_{2}}\ldots p_{r}^{\alpha_{r}}$
and $k^{'}(x):=p_{1}^{\alpha_{1}}p_{2}^{\alpha_{2}}\ldots p_{r}^{\alpha_{r}}$, by
Lemma \ref{8}, together with the Taylor expansion and the uniform convergence of $L$-series, and setting $u=q^{-s}$, we get that
\[
\begin{split}
   L_k(s, \mu)&=L_{p^{\alpha}k^{'}}(s, \mu)\\
             &=L_{pk^{'}}(s, \mu)\\
             &=L_{k^{'}}(s, \mu)\left(1-\frac{1}{|p|^{s}}\right)^{-1}\\
             &=\sum_{n=0}^{\infty}{u^{n\deg{p}}}\sum_{j=0}^{\infty}{H(j,k^{'})u^{j}}=\sum_{n=0}^{\infty}\Big(\sum_{n=i\deg{p}+j}{H(j,k^{'})}\Big){u^{n}}.
\end{split}
\]
Comparing the coefficients of above results, one obtains \[H(n,k)=\sum_{n=i\deg{p}+j}{H(j,k^{'})}.\]Therefore, by our assumption, we can get that
\[
  |H(n,k)|\leq\sum_{n=i\deg{p}+j}{|H(j,k^{'})|}\leq \sum_{d=0}^{n}{qj^{r-1}}\leq n\cdot qn^{r-1}=qn^{r}.
\]
This completes the proof.
\end{proof}
\begin{lem}\label{11}
  For any fixed polynomial $k\in\mathbb{F}_{q}[x]$ and $l\in(\mathbb{F}_{q}[T]/(k))^{\times}$ with given $n\in\mathbb{N}$, there exists a function $F(k,n):\mathbb{F}_{q}[x]\rightarrow\mathbb{R}$ which is independent of $l$ such that
  \[\sum_{\substack{f \text{ monic } \\ \deg{f}=n \\ f\equiv l (k)}}{\f}=F(k,n)+O\left(n^{\alpha}\right),\]
  where the fixed $\alpha$ depends on $k(x)\in\Fq$.
\end{lem}
\begin{proof}
Using the M\"{o}bius transform
of the von Mangolt function, we get
\[\begin{split}
\sum_{\substack{f \text{ monic } \\ \deg{f}=n \\ f\equiv l (k)}}{\f}&=-\sum_{\substack{f \text{ monic } \\ \deg{f}=n \\ f\equiv l (k)}}\sum_{\substack{m \text{ monic } \\ m\mid f  }}{\mu(m)\deg{m}}      \\
                       &=-\sum_{\substack{m \text{ monic } \\ \deg{m}\leq n  \\ (m,k)=1}}{\mu(m)\deg{m}}\sum_{\substack{f \text{ monic }\\ \deg{f}=n \\m|f\\ f\equiv l (k)}}{1},\end{split}\]
where $(l,k)=1$. Therefore, using Lemmas \ref{12} and \ref{10}, we obtain the following formula
\[\sum_{\substack{f \text{ monic } \\ \deg f=n \\ f\equiv l (k)}}{\f}=-\sum_{d=0}^{n-\deg{k}}{dq^{n-d-\deg{k}}H(d,k)}+O\left(n^{\alpha}\right).\]  Taking \[F(k,n):= -\sum_{d=0}^{n-\deg{k}}{dq^{n-d-\deg{k}}H(d,k)} \] and the desired result follows.
\end{proof}
\begin{lem}\label{13}
For any fixed polynomial $k(x)\in\Fq$ and fixed $l(x)$ such that $(l,k)\neq1$, we have the following upper bound
$$\sum_{\substack{f \text{ monic } \\ \deg{f}=n \\ f\equiv l (k)}}{\f}\leq \deg{k}.$$
\end{lem}
\begin{proof}
It is obvious that the formula
$\sum_{\substack{f \text{ monic } \\ \deg{f}=n \\ f\equiv l (k)}}{\f}$
does not vanish only if $(l, k)=p_{0}^{\beta}$, where $p_{0}(x)$ is a monic irreducible polynomial and $\beta$ is a positive integer. Then we have
\[\sum_{\substack{f \text{ monic } \\ \deg{f}=n \\ f\equiv l (k)}}{\f}=\sum_{\substack{p \text{ monic} \\ \alpha\deg{p}=n \\ p^{\alpha}\equiv l (k)}}{\deg{p}}\leq\deg p_{0}\leq \deg{k}.\]
\end{proof}
\subsection{Proof of the main result}
In this section, we shall prove Theorem \ref{1}.

{\sl Proof of Theorem \ref{1}.}
On one hand, by Definition \ref{17}, Lemmas \ref{14} and \ref{6}, we have the first equation
\[\begin{split}
   \sum_{\substack{f \text{ monic } \\ d_{f}=n \\ f\equiv l (k)}}{\f}&=\sum_{\substack{p \text{ monic} \\ \alpha\deg{p}=n \\ p^{\alpha}\equiv l (k)}}{\deg{p}}=n\sum_{\substack{p \text{ monic} \\ \deg{p}=n \\ p\equiv l (k)}}{1}+O\Big(\sum_{\substack{p \text{ monic} \\ \alpha\neq1\\\alpha\deg{p}=n}}{\deg{p}}\Big)\\
                     &=n\pi(l,k;n)+O{\left(q^{n}-n\pi(n)\right)}=n\pi(l,k;n)+O\Big(q^{\frac{n}{\Omega(n)}}\Big).
\end{split}\]
Therefore, applying Lemma \ref{11}, we get that
$$F(k,n)=n\pi(l,k;n)+O\left(q^{\frac{n}{\Omega(n)}}\right)+O\left(n^{\alpha}\right).$$
Then, let $l$ run over the ring $\mathbb{F}_{q}[x]/(k)$, by Lemma \ref{13}, we get that
\[\begin{split}
     \sum_{l\in\mathbb{F}_{q}[x]/(k)}\sum_{\substack{f \text{ monic } \\ \deg{f}=n \\ f\equiv l (k)}}{\f}&=\sum_{l\notin(\mathbb{F}_{q}[x]/(k))^{\times}}\sum_{\substack{f \text{ monic } \\ \deg{f}=n \\ f\equiv l (k)}}{\f}+\Phi(k)F(k,n)+O\left(n^{\alpha}\right)\\
                       &=\Phi(k)F(k,n)+O\left(1\right)+O\left(n^{\alpha}\right)\\
                       &=\Phi(k)n\pi(l,k;n)+O\left(n^{\alpha}\right)+O\left(q^{\frac{n}{\Omega(n)}}\right).
  \end{split}
\]
On the other hand, by Lemma \ref{14}, we have that $$\sum_{l\in\mathbb{F}_{q}[x]/(k)}\sum_{\substack{f \text{ monic } \\ \deg{f}=n \\ f\equiv l (k)}}{\f}= \sum_{\substack{f \text{ monic } \\ \deg{f}=n }}{\f} =q^{n}.$$
Finally, combining the above two equalities, we obtain that
$$\pi(l,k;n)= \frac{1}{\Phi(k)}\frac{q^{n}}{n}+O\Big(n^{\alpha}\Big)+O\Big(\frac{q^{\frac{n}{\Omega(n)}}}{n}\Big).$$
This completes the proof.
\qed


\bibliographystyle{elsarticle-num}

\end{document}